\documentclass[a4paper,11pt]{article}

\usepackage{lineno}
\usepackage{tikz}
\usetikzlibrary{arrows, automata,positioning,calc,shapes,decorations.pathreplacing,decorations.markings,shapes.misc,petri,topaths}
\usepackage{tkz-berge}
\usepackage{tkz-graph}
  \usepackage{pgfplots}
  \usepackage{array}
  \usepackage{makecell} 
  \pgfplotsset{compat=newest}
  \usetikzlibrary{plotmarks}
  \usepackage{grffile}
  \newlength\figureheight
  \newlength\figurewidth
\renewcommand{\tilde}{\widetilde}

\renewcommand{\bar}{\overline}
\renewcommand{\rho}{\varrho}

\newcommand{\mb}{\mathbb{P}}
\newcommand{\mbb}[1]{\mb\left(#1\right)}
\newcommand{\given}{~\big|~}

\usepackage{amsmath,amsfonts,amssymb,amsthm,epsfig,epstopdf,url,array}
\newtheorem{thm}{Theorem}[section]

\newtheorem{prop}[thm]{Proposition}

\definecolor{plotcolor1}{rgb}{0,0.447,0.741}
\definecolor{plotcolor2}{rgb}{0.741,0,0.447}
\definecolor{plotcolor3}{rgb}{0,0.741,0.294}
\definecolor{plotcolor4}{rgb}{0.741,0.294,0}
\definecolor{plotcoloraux}{rgb}{0.447,0.447,0.447}
\tikzset{plotstyle1/.style={color=plotcolor1,solid,line width=1.0pt}}
\tikzset{plotstyle2/.style={color=plotcolor2,densely dashed,line width=1.0pt}}
\tikzset{plotstyle3/.style={color=plotcolor3,dotted,line width=1.0pt}}
\tikzset{plotstyle4/.style={color=plotcolor4,loosely dashed,line width=1.0pt}}
\tikzset{auxlines/.style={color=plotcoloraux,solid,line width=0.5pt}}
\newcommand{\markersize}{0.7pt}
\tikzset{discretemarkers/.style={mark=*,mark options={solid},mark size=\markersize}}

\usepackage{amssymb}
\usepackage{graphicx}
\usepackage{amsfonts, indentfirst, latexsym, bm,graphicx,amsmath,amssymb,mathrsfs,verbatim,hyperref,pdfpages}
\usepackage{caption}
\usepackage{subcaption}
\usepackage{color}
\usepackage[T1]{fontenc}
\usepackage{dsfont}
\usepackage{latexsym}
\usepackage{url}
\usepackage{tikz}
\usepackage{pgfplots}
\usepackage[titletoc,title]{appendix}
\usepackage{bbm}
\usepackage{pgf}
\usetikzlibrary{arrows,snakes, automata,positioning,calc,shapes,decorations.pathreplacing}
  \colorlet{greencolor}{green!50!black}
  \colorlet{textcolor}{red}
  \colorlet{tancolor}{orange!80!black}
  \colorlet{bluecolor}{blue}

\newcommand{\Pl}{\mathbb P}
\newcommand{\E}{ \mathbb E}

\def\E{\mathbb{E}}
\def\Var{\mathrm{Var}}
  
\modulolinenumbers[5]

\title{The Role of Information in System Stability\\
with Partially Observable Servers}

\author{A. Asanjarani\footnote{a.asanjarani@uq.edu.au}, Y. Nazarathy\footnote{y.nazarathy@uq.edu.au}.\\
The University of Queensland, St Lucia, Brisbane, QLD, 4072.
}

\begin{document}

\maketitle

\begin{abstract}
We consider a simple discrete time controlled queueing system, where the controller has a choice of which server to use at each time slot and server performance varies according to a Markov modulated random environment. We explore the role of information in the system stability region. At the extreme cases of information availability, that is when there is either full information or no information, stability regions and maximally stabilizing policies are trivial. But in the more realistic cases where only the environment state of the selected server is observed, only the service successes are observed or only queue length is  observed, finding throughput maximizing control laws is a challenge.  To handle these situations, we devise a Partially Observable Markov Decision Process (POMDP) formulation of the problem and illustrate properties of its solution. We further model the system under given decision rules, using Quasi-Birth-and-Death (QBD) structure to find a matrix analytic expression for the stability bound. We use this formulation to illustrate how the stability region grows as the number of controller belief states increases. 

Our focus in this paper is on the simple case of two servers where the environment of each is modulated according to a two-state Markov chain. As simple as this case seems, there appear to be no closed form descriptions of the stability region under the various regimes considered. Our numerical approximations to the POMDP Bellman equations and the numerical solutions of the QBDs hint at a variety of structural results.
\end{abstract}

\textbf{keyword:}
system stability, POMDP, control, queueing systems, optimal, Bellman equations, QBD, information, Markov models.

\linenumbers

\section{Introduction}
Performance evaluation and control of queueing systems subject to randomly varying environments is an area of research that has received much attention during the past few decades (see for example, \cite{cecchi2016nearly}, \cite{johnston2006opportunistic}, \cite{SKRS2008} and references there-in). This is because numerous situations arise in practice where a controller needs to decide how to best utilise resources, and these are often subject to changing conditions. Examples of such situations arise in wireless communication, supply chain logistics, health care, manufacturing and transportation. In all these situations, it is very common for service rates to vary in a not fully predictable manner. Using Markovian random environments has often been a natural modelling choice due to the tractability and general applicability of Markov models. See for example Section A.1 in \cite{MeynBook} for a general discussion on the ubiquity of Markov models.

The bulk of the literature dealing with performance evaluation and control of these types of problems, has considered the situation where the state of the underlying random environment is observable. In such a situation, it is already a non-trivial task to carry out explicit system performance analysis (see \cite{kellaWhitt} as an example). Further, finding optimal or even merely stabilizing control is typically a formidable achievement (see for example \cite{BacelliMakowski}, \cite{TassiulaEphremides1993} or the more recent \cite{HalabianEtAl2014}). But in practice, the actual environment state is often not a directly observed quantity, or is at best only partially observable. The situation is further complicated when control decisions do not only affect instantaneous rewards, but also the observation made. In classic linear-quadratic optimal control settings (e.g. Part III of \cite{whittleBook}), the certainty equivalence principle allows to decouple state estimation based on observations and control decisions. However, in more complicated settings such as what we consider here, certainty equivalence almost certainly doesn't hold.

In this paper we augment the body of literature dealing with exploration vs. exploitation trade-offs in systems where a controller needs to choose a server (channel/resource/bandit) at any given time, and the choice influences both the immediate reward (service success) and the information obtained. A general class of such problems, denoted Reward Observing Restless Multi Armed Bandits (RORMAB), is outlined in \cite{chapterKuhnNaz2015}, where much previous literature is surveyed. Key contributions in this area are \cite{Koole} and the more recent \cite{LZ2010}. The former finds the structure of optimal policies from first principles. The latter, generalizes the setting and utilizes the celebrated Whittle Index, \cite{Whittle1988} for such a partially observable case.   Related recent results dealing with RORMAB problems are in \cite{larranaga2016asymptotically} and \cite{larranaga2014index}. Of further interest is the latest rigorous account on asymptotic optimality of the Whittle index, \cite{verloop2014asymptotically}.

Our focus in this paper, is on a controlled queueing systems, where server environments vary and the controller (choosing servers) only observes partial information. Our aim is to explore the role of information in system stability. For this, we devise what is perhaps the simplest non-trivial model possible: a single discrete time queue is served by either Server~$1$ or Server~$2$ where each server environment is an independent two-state Markov chain. A controller having (potentially) only partial state information, selects one of the two servers at each time instance. 

The role of information is explored by considering different observation schemes. At one extreme, the controller has full information of the servers' environment states. At the other extreme, the controller is completely unaware of the servers' environment states. Obviously the stability region of the system in the latter situation is a subset of the former. Our contribution is in considering additional more realistic observation schemes. One such scheme is a situation where the controller only observes the state of the server currently chosen. This type of situation has been widely studied in some of the references mentioned above and surveyed in \cite{chapterKuhnNaz2015}, but most of the literature dealing with this situation does not consider a queue. A more constrained scenario is one where the controller only observes the success/failure of service (from the server chosen) at every time slot. Such a partial observability situation was recently introduced in \cite{nazarathy2015challenge} in the context of stability and analysed in \cite{meshram2016whittle} with respect to the Whittle index. In \cite{li2013network}, stability of a related multi-server system was analysed. 

An additional observation scheme that we consider is one where the server is only aware of the queue size process. In (non-degenerate) continuous time systems, such an observation scheme is identical to the former scheme. But an artefact of our discrete time model is that such a scheme reveals less information to the controller (this is due to the fact that both an arrival and a departure may occur simultaneously, going unnoticed by the controller).

With the introduction of the five observation schemes mentioned above, this work takes first steps to analyse the effect of information on the achievable stability region. A controller of such a system makes use of a belief state implementation. We put forward (simple) explicit belief state update recursions for each of the observation schemes. These are then embedded in Bellman equations describing optimal solutions of associated Partially Observable Markov Decision Processes (POMDP). Numerical solution of the POMDPs then yields insight on structural properties and achievable stability regions. By construction, two-state Markov server environments are more predictable when the mixing times of the Markov chains increase. We quantify this use of channel-memory, through numerical and analytic results.

It is often the case that MDPs (or POMDPs) associated with queueing models, can be cast as QBDs once a class of control policies is found. See for example \cite{meszaros2014markov}. We follow this paradigm in the current paper and present a detailed QBD model of the system. The virtue of our QBD based model is that we are able to quantify the effect of a finite state controller on the achievable stability region whose upper bound is given by an elegant matrix analytic expression.

The remainder of this paper is structured as follows.
In Section~\ref{sec:model}, we introduce the system model and different observation schemes. In Section~\ref{sec:belief}, we put forward recursions for belief state updates for the non-trivial observation schemes. In Section~\ref{sec:pomdp}, we present the myopic policy and the Bellman equations for different observation schemes and present findings from a numerical investigation. In Section \ref{sec:finiteQBD},  we construct a QBD representation of the system, find the stability criterion and put forward the numerical results which are matched with the results of Bellman equations of Section~\ref{sec:pomdp}. We conclude in Section~\ref{sec:outlook}.

\section{System Model}
\label{sec:model}

\begin{figure}[h]
\begin{center}
\begin{tikzpicture}[xscale=1,yscale=0.75]

 \def\hh{0.6}
\draw [<-, thick]  (1.5,4-\hh)-- (1,4-\hh)  node[left] {$E(t)$};
\draw (0.4+\hh+\hh,3.55-\hh) -- (1.5+\hh+\hh,3.55-\hh) -- (1.5+\hh+\hh,4.35-\hh) -- (0.4+\hh+\hh,4.35-\hh);

\draw[ ultra thick]  (1.4+\hh+\hh,0.08+3) --(1.4+\hh+\hh,0.7+3);
\draw[ ultra thick]  (1.25+\hh+\hh,0.08+3) --(1.25+\hh+\hh,0.7+3);
\draw[ ultra thick]  (1.1+\hh+\hh,0.08+3) --(1.1+\hh+\hh,0.7+3);
\draw[ ultra thick]  (0.95+\hh+\hh,0.08+3) --(0.95+\hh+\hh,0.7+3);
\node at (2.2, 2.6){$Q(t)$};

\draw (7,0.6) rectangle (9.6,2.5);

\begin{scope}[scale=0.75,transform shape]
 \Vertex[x=12,y=1.9]{1}
  \Vertex[x=10,y=1.9]{0}
  \tikzstyle{EdgeStyle}=[bend left,post]

  \Edge(0)(1)
  \Edge (1)(0)
  \node at (10.2, 2.6){$I^{(2)}_0(t)$};
  \node at (11.8, 2.6) {$I^{(2)}_1(t)$};
     \node at (11, 1.2) {$X_2(t)$};      
                       
\end{scope}
\draw[->, thick]  (9.9,1.45) -- (10.5,1.45)node[right] {$I_{X_2(t)}^{(2)}(t)$};

 \def\h{4}
 \draw (7,0.6+\h) rectangle (9.6,2.5+\h);

\begin{scope}[scale=0.75,transform shape]
 \Vertex[x=12,y=3.2+\h]{1}
  \Vertex[x=10,y=3.2+\h]{0}
  \tikzstyle{EdgeStyle}=[bend left,post]

  \Edge(0)(1)
  \Edge (1)(0)
    \node at (10.2, 4+\h){$I^{(1)}_0(t)$};
  \node at (11.8, 4+\h) {$I^{(1)}_1(t)$};
     \node at (11, 2.5+\h) {$X_1(t)$};     
      
   \end{scope}
\draw[->, thick]  (9.8,1.45+\h) -- (10.4,1.45+\h)node[right] {$I_{X_1(t)}^{(1)}(t)$};

\node at (4.8,3.5) [cloud, draw,cloud puffs=10,cloud puff arc=120, aspect=0.75, inner ysep=0.9em,fill=gray!10] {$\pi\rightarrow U(t)$};

\draw[-,thick, dashed]  (2.8,3.5) --(3.5,3.5);
\draw[->,thick, dashed]  (6.2,4.1) --(7,4.8);
\draw[->,thick, dashed]  (6.2,3) --(7,2.3);


\draw [->, dotted,line width=1.0pt,plotcolor2] (9.2,3.45) to
(6.5,3.45); 
\node at (8.1,3.75){\color{plotcolor2}$Y(t)$};
\end{tikzpicture}
\caption{A controller operating under a decision rule, $\pi$, decides at each time step, $t$, if to use server $U(t)=1$ or $U(t)=2$ based on previously observed information, $Y(t-1), Y(t-2),\ldots$. The server environments, $X_i(t)$ are Markov modulated.\\
\label{fig:2servermodelPrimitives}}
\end{center}
\end{figure}
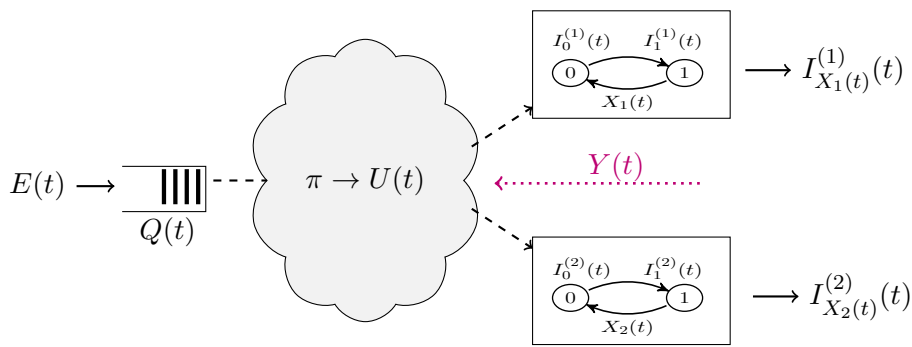

Consider a situation as depicted in Figure~\ref{fig:2servermodelPrimitives}.
Jobs arrive  into the queue, $Q(t)$,  and are potentially served by one of two servers $j=1$ or $j=2$, according to some control policy. The system is operating in discrete time steps $t =0,1,\dots$, where in each time step, the following sequence of events occur:
\begin{enumerate}
\item An arrival occurs as indicated through $E(t)$: $E(t)=1$ indicates an arrival and otherwise $E(t)=0$.
\item The environments of the servers update  from $\big(X_1(t-1), \, X_2(t-1) \big)$ to $\big(X_1(t), \, X_2(t) \big)$, autonomously. That is, the updating of the environment is not influenced by arrivals, queue length and controller choice.
\item A control decision $u=U(t)$ of which server to select is made based on observations in previous time steps, denoted by $Y(t-1), Y(t-2),\ldots$. This is through a decision rule, $\pi$. We consider different observation schemes as described below.
\item The control action is executed and the queue length is updated as follows:
\[
Q(t+1) = Q(t) + E(t) - I(t),
\qquad
\mbox{with}
\qquad I(t) = I^{(u)}_{X_{u}(t)}(t).
\]
Here,  $I(t)$ indicates whether there was a service success or not. It is constructed from the primitive sequences,
\[
\big\{ I_{i}^{(j)}(t),~t=0,1,2,\ldots \big\},
\]
for servers $j=1,2$ and environment states $i=0,1$. Note that when $Q(t)=0$, we notionally assume that $U(t)=u=0$, indicating ``no action'' and in this case denote $I^{(0)}_i \equiv 0$ for $i=0,1$.
\item The observation of $Y(t)$ is made and is used in subsequent time steps.
\end{enumerate}

The sequence of events (1)--(5) as above repeats in every time step and fully defines the evolution law of the system.  We consider the following distinct observation schemes:\\

\noindent
\textbf{(I) Full observation}:  The controller knows the state of both servers all the time. In this case
\[
Y(t) = \big(X_1(t), \, X_2(t) \big),
\]
and further, the sequence of steps above is slightly modified with step (5) taking place between steps (2) and (3) and the policy at step 3 being
\[
u=U(t) = \pi\big(Y(t)\big).
\]

\noindent
\textbf{(II) State observation:} The controller observes the state of the selected server at time $t$, but does not observe the other server at that time. Hence 
\[
Y(t) = 
\begin{cases}
\big(X_1(t), \emptyset  \big) & \mbox{if}~~u=1, \\
\big(\emptyset, X_2(t)  \big) & \mbox{if}~~u=2. \\
\end{cases}
\]

\noindent
 \textbf{(III) Output  observation:} The controller observes the success or failure of outputs of the server selected (but gains no information about the other server at that time). Hence
\[
Y(t) = 
\begin{cases}
\big(I^{(1)}_{X_1(t)}(t), \emptyset  \big) & \mbox{if}~~u=1, \\
\big(\emptyset, I^{(2)}_{X_2(t)}(t)  \big) & \mbox{if}~~u=2. \\
\end{cases}
\]
\noindent
\textbf{(IV) Queue observation:} The controller only observes the queue length, $Q(t)$, and can thus utilize the differences
\[
\Delta Q(t) = Q(t+1) - Q(t) = E(t) - I(t).
\]
Note that since the system is operating in discrete time, there is some loss of information compared to case III : If $\Delta Q(t) = 1$ or $\Delta Q(t) = -1$, then it is clear that $I(t) = 0$ or $I(t) = 1$, respectively. But if $\Delta Q(t) = 0$, then since the controller does not observe $E(t)$, there is not a definitive indication of $I(t)$.\\

\noindent
\textbf{(V) No  observation:} We assume the controller does not observe anything. Nonetheless, as with the other cases, the controller knows the system parameters as described below.

\vspace{10pt}

We consider the simplest non-trivial probably model for the primitives. These are $E(t)$, $I_i^{(j)}(t)$ and the environment processes, $X_j(t)$, all assumed mutually independent. The arrivals, $E(t)$, are an i.i.d. sequence of Bernoulli random variables, each with probability of success $\lambda$. The service success indicators, for each server $j=1,2$ and state $i=0,1$, denoted by $\{I_i^{(j)}(t), ~t=0,1,2,\ldots\}$, are each an i.i.d. sequence with
\[
I_i^{(j)}(t) ~\sim~ \mbox{Bernoulli}\big(\mu_i^{(j)}\big).
\]
Moreover, we assume
\begin{equation}\label{eq:assumption}
\mu_0^{(2)}\leq \mu_0^{(1)}< \mu_1^{(1)}\leq \mu_1^{(2)}.
\end{equation}
Hence states $i=1$ for both servers are better than states $i=0$. Further, the spread of the chance of success for Server~$2$ is greater or equal to that for Server~$1$.

For the environment processes, we restrict attention to a two-state Markov chain, sometimes referred to as a Gilbert--Elliot channel \cite{SKRS2008}. We denote the probability transition matrix for server $j$ as:
\begin{equation}\label{eq:prob.matrix}
P^{(j)}
=
\left[
\begin{array}{cc}
\bar{p}_j  & p_j \\
q_j & \bar{q}_j
\end{array}
\right]
=
\left[
\begin{array}{cc}
1- \gamma_j \, \bar{\rho}_j& \gamma_j \, \bar{\rho}_j \\
 \bar{\gamma}_j\, \bar{\rho}_j&1- \bar{\gamma}_j\, \bar{\rho}_j
 \end{array}
\right],
\end{equation}
with $\bar{x} :=1-x$. In the sequel, we omit the server index $j$ from the individual parameters of $P^{(j)}$.   A standard parametrization of this Markov chain uses transition probabilities $p, q \in [0,1]$.
Alternatively, we may specify the stationary probability of being in state $1$, denoted by $\gamma \in [0,1]$,
together with the second eigenvalue of $P$, denoted by 
$\rho\in \big[ 1- \min(\gamma^{-1},\, \bar{\gamma}^{-1} \big),~ 1\big]$. 

Using this parametrization, $\rho$ quantifies the time-dependence of the chain; 
when $\rho= 0$ the chain is i.i.d., otherwise there is memory. If $\rho>0$ then environment states are positively correlated, otherwise they are negatively correlated. Our numerical examples in this paper deal with positive correlation as it is often the more reasonable model for channel memory.
The relationship between the $(\gamma,\rho)$ parametrization and the $(p,q)$ parametrization is given by $p = \gamma \, \bar{\rho}$, $q = \bar{\gamma}\,\bar{\rho}$, $\gamma =p/({p+q})$, and $\rho = 1-p-q$. See Figure~\ref{fig:2servermodel}.
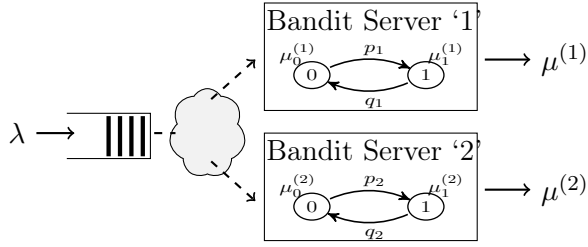
\begin{figure}[h]
\begin{center}
\hspace{5mm}\begin{tikzpicture}[xscale=1,yscale=0.75]

 \def\hh{0.5}
\draw [<-, thick]  (2.5,3-\hh)-- (2,3-\hh)  node[left] {$\lambda$};
\draw (1.4+\hh+\hh,2.55-\hh) -- (2.5+\hh+\hh,2.55-\hh) -- (2.5+\hh+\hh,3.35-\hh) -- (1.4+\hh+\hh,3.35-\hh);

\draw[ ultra thick]  (2.4+\hh+\hh,0.09+2.5-0.5) --(2.4+\hh+\hh,0.8+2.5-0.5);
\draw[ ultra thick]  (2.25+\hh+\hh,0.09+2.5-0.5) --(2.25+\hh+\hh,0.8+2.5-0.5);
\draw[ ultra thick]  (2.1+\hh+\hh,0.09+2.5-0.5) --(2.1+\hh+\hh,0.8+2.5-0.5);
\draw[ ultra thick]  (1.95+\hh+\hh,0.09+2.5-0.5) --(1.95+\hh+\hh,0.8+2.5-0.5);

\draw (5,0.6) rectangle (7.8,2.5);
\node at (6.45, 2.2) {Bandit Server `2'};

\begin{scope}[scale=0.75,transform shape]
 \Vertex[x=9.5,y=1.6]{1}
  \Vertex[x=7.5,y=1.6]{0}
  \tikzstyle{EdgeStyle}=[bend left,post]

  \Edge(0)(1)
  \Edge (1)(0)
   \node at (7.25, 2.1) {$\mu_0^{(2)}$};
      \node at (9.85, 2.1) {$\mu_1^{(2)}$};
            \node at (8.6, 0.95) {$q_{2}$};
                        \node at (8.6,2.2) {$p_{2}$};
\end{scope}
\draw[->, thick]  (7.9,1.5) -- (8.5,1.5)node[right] {$\mu^{(2)}$};

 \def\h{2.3}
 \draw (5,0.6+\h) rectangle (7.8,2.5+\h);
\node at (6.45, 2.2+\h) {Bandit Server `1'};

\begin{scope}[scale=0.75,transform shape]
 \Vertex[x=9.5,y=4.7]{1}
  \Vertex[x=7.5,y=4.7]{0}
  \tikzstyle{EdgeStyle}=[bend left,post]

  \Edge(0)(1)
  \Edge (1)(0)
   \node at (7.25, 5.2) {$\mu_0^{(1)}$};
      \node at (9.85, 5.2) {$\mu_1^{(1)}$};
            \node at (8.6, 4.05) {$q_{1}$};
                        \node at (8.6,5.3) {$p_{1}$};
\end{scope}
\draw[->, thick]  (7.9,1.5+\h) -- (8.5,1.5+\h)node[right] {$\mu^{(1)}$};

\node at (4.25,2.5) [cloud, draw,cloud puffs=7,cloud puff arc=120, aspect=0.75, inner ysep=1em,fill=gray!10] {};

\draw[-,thick, dashed]  (3.55,2.5) --(3.83,2.5);
\draw[->,thick, dashed]  (4.3,3.1) --(4.88,3.8);
\draw[->,thick, dashed]  (4.3,2) --(4.88,1.3);


\end{tikzpicture}
\caption{Parameters of the system.\label{fig:2servermodel}}
\end{center}
\end{figure}

It is instructive to consider the long term behaviour of $I_{X_j(t)}^{(j)}(t)$ for $j=1,2$ by assuming the sequence $\{X_j(t),~t=0,1,2,\ldots\}$ is stationary and thus each $X_j(t)$ is Bernoulli distributed with parameter $\gamma_j$.  In this case
\begin{align*}
\E [I_{X_j(t)}^{(j)}(t)] & = \bar \gamma_j\,\mu_0^{(j)}+\gamma_j \,\mu_1^{(j)}\,,\\
\Var [I_{X_j(t)}^{(j)}(t)] & = \bar \gamma_j\,\mu_0^{(j)}\,\bar \mu_0^{(j)}+
\gamma_j \,\mu_1^{(j)}\,\bar \mu_1^{(j)}+\gamma_j\, \bar\gamma_j\,\big(\mu_1^{(j)}-\mu_0^{(j)}\big)^2\,.
\end{align*}
These quantities  are useful for obtaining a rough handle on the performance of the system.

The {\em stability region} of the system is defined as the set of arrival rates for which there exist a decision rule, $\pi$,  under which the Markov chain describing the system is positive recurrent\footnote{There may be different Markov chain representations of this system. One concrete description is $Z(t)$ in Section \ref{sec:finiteQBD}.}. Note that this Markov chain, encapsulates $Q(t)$, but may also include additional state variables. We postulate that stability regions are of the form
\[
\{ \lambda ~:~ \lambda < \mu^*\},
\]
where the {\em stability bound} $\mu^*$ varies according to the observation schemes I--V above and is identical to the maximal throughput rate that may be obtained in a system without a queue (but rather with an infinite supply of jobs). Further, by the construction of the information observation schemes, we have 
\begin{equation}\label{eq:relation}
\mu^*_\text{no} \le  \mu^*_\text{queue} \le \mu^*_\text{output} \le \mu^*_\text{state} \le \mu^*_\text{full}\,,
\end{equation}
where $\mu^*_\text{no}$ corresponds to case~V, $\mu^*_\text{queue}$ corresponds to case~IV and so forth.

The lower and upper bounds, $\mu^*_\text{no}$ and $\mu^*_\text{full}$, are easily obtained as we describe now. However, the other cases are more complicated and are analysed in the sections that follow.
For the lower  and upper bounds we have
\begin{align*}
\mu^*_\text{no} &= \max \big\{ \E [I_{X_1(t)}^{(1)}(t)],~ \E [I_{X_2(t)}^{(2)}(t)] \big\}, \\
\mu^*_\text{full} &= \bar \gamma_1 \bar \gamma_2 \mu_0^{(1)}+\gamma_2 \mu_1^{(2)}+ \gamma_1 \bar \gamma_2 \mu_1^{(1)}.
\end{align*}
The lower bound, $\mu^*_\text{no}$ is trivially achieved with a control policy that always uses the server with the higher mean throughput. The upper bound, is achieved with a control policy that uses the best server at any given time. Under the ordering in~\eqref{eq:assumption}, the throughput in this case is calculated as follows: If both servers are in state 0, then since $\mu_0^{(2)}\leq \mu_0^{(1)}$, the controller selects Server~$1$. This situation occurs at a long term proportion, $\bar \gamma_1 \bar \gamma_2$,  hence we obtain the first term of $\mu^*_\text{full}$. The other terms of $\mu^*_\text{full}$ are obtained with a similar argument. Note that when $\gamma_1=\gamma_2=\gamma$ and $\mu_i^{(1)} = \mu_i^{(2)}=\mu_i$ for $i=0,1$, the expression is reduced to  $\mu^*_\text{full} = \overline{\gamma}^2 \mu_0 + (1- \overline{\gamma}^2) \mu_1$ and can be obtained by a Binomial argument.

As a benchmark numerical case, all the examples we present use
\begin{equation}
\label{eq:111}
\gamma=0.5\,,~\mu_0 = 0.2\,, ~\mu_1 = 0.8\,,
\end{equation}
for both servers. Under these parameters
\[
\mu^*_\text{no} = 0.5\,, \qquad \text{and} \qquad \mu^*_\text{full} = 0.65\,.
\]
Hence in the examples that follow, we explore how $\mu^*_\text{queue}$, $\mu^*_\text{output}$ and $\mu^*_\text{state}$ vary within the interval $[0.5,0.65]$ as $\rho_j$, $j=1,2$ varies.

\section{Belief States}
\label{sec:belief}

In implementing a controller for each of the observation schemes, the use of {\em belief states} reduces both the complexity of the controller and the related analysis. The idea is to summarize the history of observations, $Y(t-1), Y(t-2),\ldots\,$,  into {\em sufficient statistics} that are updated by the controller. For our model, a natural choice for the belief state of server $j$ is
\[
\omega_j(t) = \mb \big( X_j(t) = 1 \, | \, \mbox{Prior knowledge to time } t \big).
\]
As we describe now, it is a simple matter to recursively update this sequence in a Bayesian manner. Denoting $\omega_j(t)$ by $\omega$, the believed chance of  success  is 
$$
r(\omega) :=  \bar\omega \mu_0 + \omega \mu_1.
$$
The updating algorithms (different for each observation scheme) make use of the following:
\begin{equation}\label{eq:belief}
\begin{array}{cc}
\tau_n(\omega):= \omega \rho + \gamma \bar \rho\,,\qquad
\tau_f(\omega):=\frac{\bar{q}\,\bar\mu_1\,\omega+p\, \bar\mu_0\,\bar\omega}{\bar r(\omega)}\,,\\
\\
\qquad\tau_s(\omega):=\frac{\bar{q}\,\mu_1\,\omega+p\,\mu_0\,\bar\omega}{r(\omega)}, \quad 
\tau_c(\omega):=\lambda \tau_s(\omega) + \bar \lambda \tau_f(\omega).
\end{array}
\end{equation}
Note that in the above, superscripts $j$ are omitted for clarity. The probabilistic meaning of these functions is described in the sequel. These are used to define recursions for updating the belief state. Each observation scheme entails a different type of recursion:

\vspace{10pt}
\noindent
\textbf{(II) State observation:} 
\begin{align*}
\big(\omega_1(t+1),\omega_2(t+1) \big) &=
\begin{cases}
\big( X_1(t), \tau_n^{(2)}\big(\omega_2(t)\big), & U(t) = 1, \\
\big(\tau_n^{(1)}\big(\omega_1(t)\big),X_2(t)\big) ,  & U(t) = 2. \\
\end{cases} 
\end{align*}

\noindent
 \textbf{(III) Output  observation:} 
$$
 \big(\omega_1(t+1),\omega_2(t+1) \big) = \left \{
  \begin{aligned}
    &\big(\tau_f^{(1)}(\omega_1(t)),\tau_n^{(2)}\big(\omega_2(t)\big)\big) , && \quad I^{(1)}_{X_1(t)}(t)=0, \\
    &  && \quad && \quad U(t) = 1,\\
    &\big(\tau_s^{(1)}(\omega_1(t)),\tau_n^{(2)}\big(\omega_2(t)\big)\big), && \quad I^{(1)}_{X_1(t)}(t)=1, \\
\\
     &\big( \tau_n^{(1)}\big(\omega_1(t)\big),\tau_f^{(2)}(\omega_2(t))\big), && \quad I^{(2)}_{X_2(t)}(t)=0 ,\\
    &  && \quad && \quad U(t) = 2.\\
    &\big( \tau_n^{(1)}\big(\omega_1(t)\big),\tau_s^{(2)}(\omega_2(t))\big), && \quad I^{(2)}_{X_2(t)}(t)=1 ,\\
  \end{aligned} \right.
$$ 
\noindent
\textbf{(IV) Queue observation:} 
$$
  \big(\omega_1(t+1),\omega_2(t+1) \big) = \left \{
  \begin{aligned}
    &\big(\tau_f^{(1)}(\omega_1(t)), \tau_n^{(2)}\big(\omega_2(t)\big)\big), && \quad \Delta Q(t)=1, \\
    &\big(\tau_c^{(1)}(\omega_1(t)),\tau_n^{(2)}\big(\omega_2(t)\big)\big), && \quad \Delta Q(t)=0, && \quad U(t)=1,\\
    &\big(\tau_s^{(1)}(\omega_1(t)),\tau_n^{(2)}\big(\omega_2(t)\big)\big), && \quad \Delta Q(t)=-1, \\
\\
     &\big(\tau_n^{(1)}\big(\omega_1(t)\big),\tau_f^{(2)}(\omega_2(t))\big), && \quad \Delta Q(t)=1,  \\
    & \big(\tau_n^{(1)}\big(\omega_1(t)\big), \tau_c^{(2)}(\omega_2(t))\big), && \quad \Delta Q(t)=0, && \quad U(t)=2.\\
    &\big(\tau_n^{(1)}\big(\omega_1(t)\big),\tau_s^{(2)}(\omega_2(t))\big), && \quad \Delta Q(t)=-1, \\
  \end{aligned} \right.
$$ 
Upon applying the recursions above, we indeed track the belief state as needed:
\begin{prop}
For each of the observation schemes, assume that at $t=0$, $\omega_j(0) = \Pl\big(X_j(0) = 1\big)$. Then upon implementing the recursion above, based on the observations, it holds that
\[
\omega_j(t) = \mb \big( X_j(t) = 1 \, | Y(t), Y(t-1),\ldots,Y(0) \big),\qquad t=1,2,\ldots\,.
\]
\end{prop}
\begin{proof}
The proof and derivation of the operators in \eqref{eq:belief}  follows from elementary conditional probabilities and induction. We illustrate this for the output observation case here. It holds that
\begin{align*}
&\Pl\Big({X(t)=1 \given I(t-1)=0}\Big)
=\frac{\mbb{X(t)=1,I(t-1)=0}}{\mbb{I(t-1)=0}} \\
&=
\frac{\mbb{X(t)=1,I=0\given X=1}\mbb{X=1} + \mbb{X(t)=1,I=0\given X=0}\mbb{X=0}}
{\mbb{I=0\given X=1}\mbb{X=1}
+\mbb{I=0\given X=0}\mbb{X=0}},
\end{align*}
where we denote $X=X(t-1)$ and $I=I(t-1)$. Since $X(t)$ and $I(t-1)$ are conditionally independent given $X(t-1)$, the above numerator can be written as:
\begin{align*}
&\mbb{X(t)=1\given X=1}\mbb{I=0\given X=1}\mbb{X=1}\\
&+\mbb{X(t)=1\given X=0}\mbb{I=0\given X=0}\mbb{X=0}=\bar{q}\,\bar\mu_1\,\omega+p\, \bar\mu_0\,\bar\omega.
\end{align*}
Similarly, for the denominator we have:
$$\mbb{I=0\given X=1}\mbb{X=1}
+\mbb{I=0\given X=0}\mbb{X=0}= \bar \mu_1   \omega+ \bar\mu_0 \bar \omega,
$$
which is equal to $\bar r (\omega)$. Hence as expected, we find that
\[
 \Pl\Big({X(t)=1 \given I(t-1)=0}\Big)= \tau_f(\omega).
\]
 The derivation of $\tau_n(\omega)$ and $\tau_s(\omega)$ follows similar lines. For the queue observation case, notice that
\begin{align*}
\Pl\Big({X(t)=1\given \Delta Q(t)=1}\Big) &= \tau_f(\omega)\,,\qquad
\Pl\Big({X(t)=1\given \Delta Q(t)=-1}\Big) = \tau_s(\omega)\,,\\
\Pl\Big({X(t)=1\given \Delta Q(t)=0}\Big)&=\lambda \tau_s(\omega) + \bar \lambda \tau_f(\omega)=\tau_c(\omega).
\end{align*}
The state observation case follows similar lines.
\end{proof}

Note that the fixed point of $\tau_n$ is the stationary probability $\gamma$.
The fixed points of $\tau_f$ and $\tau_s$ are also of interest.
When $\rho\neq 0$ and $\mu_0\neq \mu_1$,
$\tau_f$ and $\tau_s$ are (real) hyperbolic M\"obius transformations of the form $(a\omega+b)/(c\omega+d)$ for $\omega\in[0,1]$.
As such, they each have two distinct fixed points, one stable and one unstable.
Here, excluding trivialities where $p, q \in\{0,1\}$,
the stable fixed point of each lies in $(0,1)$ and is of the form
$\big(a-d+\sqrt{(a-d)^2+4 bc}\big)/2c$
(see also Lemma 2 and  3 of \cite{Macphee1995}).
For $\tau_f$, we have $a=\bar{q}\,\bar\mu_1-p\,\bar\mu_0$, $b=p\,\bar\mu_0$, $c=\bar\mu_1-\bar\mu_0$, and $d=\bar\mu_0$. Fixed point of $\tau_s$ comes from the same formula by replacing $\bar\mu_i$ by~$\mu_i$.

Denote by $\omega_i^{(j)}$ the stable fixed point of $\tau_i$ for  $i=0,1$ and $j= 1,2$. Then 
$$
\Omega=\Omega_1 \times \Omega_2 \subset [0,1]\times [0,1],
$$
where we put
$$
\Omega_j =[\min(\omega_0^{(j)},\omega_1^{(j)}),\,\max(\omega_0^{(j)},\omega_1^{(j)})],
$$
is the \textit{belief state space} and  the limit of any infinite subsequence of the mappings $\tau_n$, $\tau_f$, $\tau_s$  and $\tau_c\,$(for $\omega_1, \omega_2 \in [0,1]$) lies within $\Omega$, see \cite{nazarathy2015challenge} for more details.

\section{Maximal Throughput}
\label{sec:pomdp}

Having defined sufficient statistics for the belief state and their evolution, the problem of finding a maximally stabilizing control can be posed as a Partially Observable Markov Decision Process (POMDP), see for example \cite {bauerle2011markov} or the historical reference \cite{SS1973}. The objective for the POMDP is
\[
\mu^* = \sup_\pi \liminf_{T \to \infty} \frac{1}{T} \E_\pi \Big[\sum_{t=0}^{T-1} I(t)\Big],
\]
where $U(t) = \pi\big(\omega_1(t),\omega_2(t)\big)$ influences the $I(t)$ as outlined in Section~2. A formal treatment of the POMDP, relating it to the maximal stability region of the system can be carried out. We now first introduce the myopic policy for the POMDP and then move onto optimality equations.

\vspace{10pt}
\noindent
{\bf The Myopic Policy}

One specific policy is the {\em myopic policy} given by: 
\begin{equation}
\label{eq:99}
\pi(\omega_1, \omega_2) =
\begin{cases}
\mbox{Server 2} &
\mbox{if}\qquad
\omega_2 \geq \frac{\mu_1^{(1)}-\mu_0^{(1)}}{\mu_1^{(2)}-\mu_0^{(2)}}\,\omega_1+\frac{\mu_0^{(1)}-\mu_0^{(2)}}{\mu_1^{(2)}-\mu_0^{(2)}},\\
\mbox{Server 1} & 
\mbox{if}
\qquad
\mbox{otherwise}.
\end{cases}
\end{equation}
The affine threshold in this policy is obtained by comparing the immediate expected mean throughput for any given pair $(\omega_1,\omega_2)$ and choosing the server that maximizes it. Such a policy is attractive in that it is easy to implement. Further, when the servers are symmetric (all parameters are identical), it holds from symmetry that it is optimal. In this case it can be represented as
\[
\pi(\omega_1,\omega_2) = \mbox{argmax}_{i=1,2} ~\omega_i,
\]
and we refer to it as the {\em symmetric myopic policy}.

\vspace{10pt}
\noindent
{\bf Simulation Result}

Figure~\ref{fig:convergence} demonstrates results obtained through a Monte Carlo simulation\footnote{Simulation details: We run the process for  $t=5,000,000$ time units, using common random numbers for each run and recording average throughput.} of the model for observation schemes (I)--(V). We use the parameters in \eqref{eq:111} and vary $\rho$ with $\rho_1=\rho_2 = \rho$ in the range $[0,1]$ with steps of $0.01$. The policy used is the symmetric myopic policy and is optimal since the servers are identical.

\begin{figure}[h]
\setlength\figureheight{2.0cm} \setlength\figurewidth{3.0cm}

\includegraphics{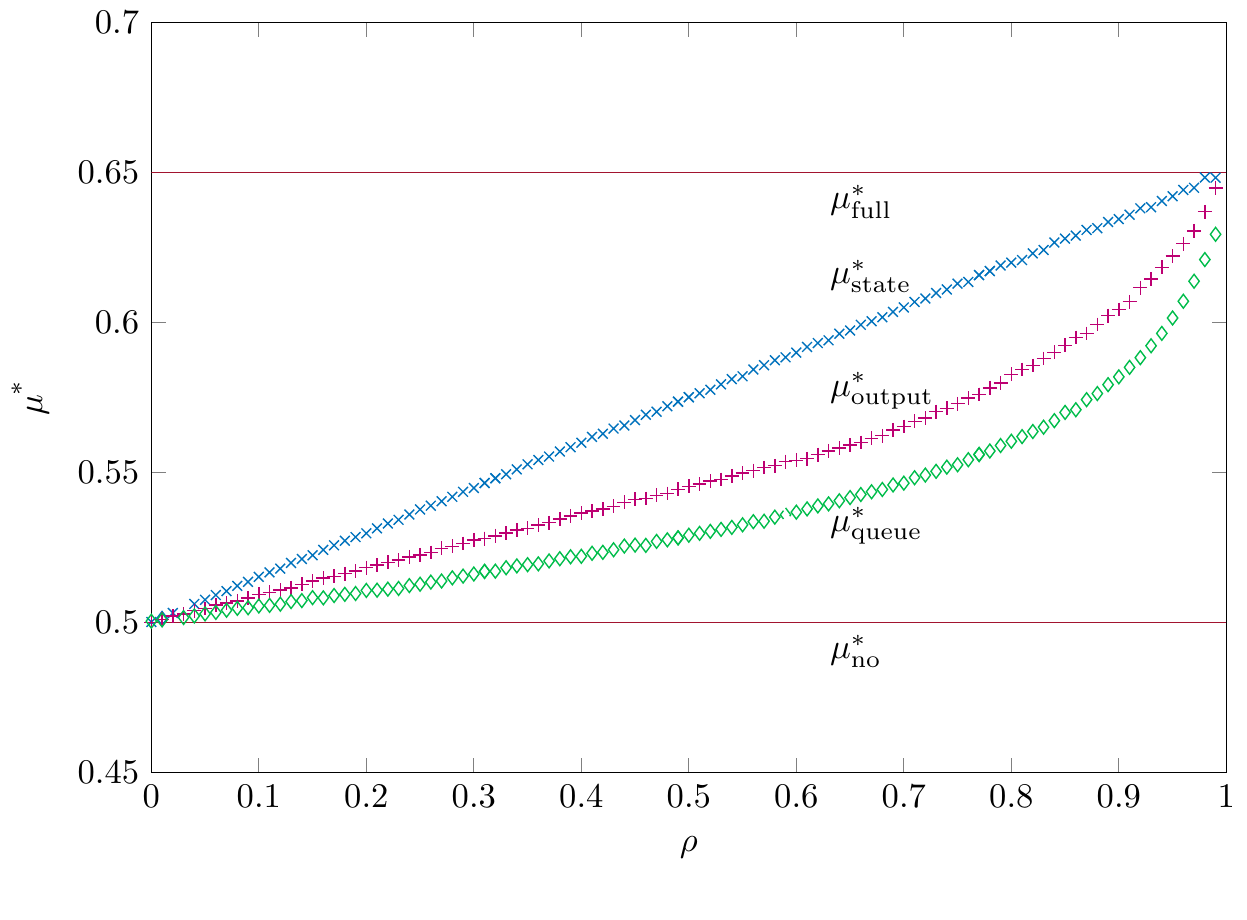}
\begin{center}
\vspace{-2mm}
\caption{The stability bound is displayed as a function of $\rho$ for the various observation schemes. This plot is based on simulation results using the parameters in \eqref{eq:111} and $\rho_1=\rho_2 = \rho$.
\label{fig:convergence}
} 
\end{center}
\end{figure}
As we see from the figure, the ordering \eqref{eq:relation} appears to hold. Further, at the i.i.d. case $\rho=0$, historical observations are not useful and the throughput of all observations schemes, except for full observation, is at $0.5$. At the other extreme, when $\rho \to 1$, we have that the state observation scheme converges to a throughput identical to that of the full observation scheme. This is because at that regime, the server state environments rarely change. Thus, from a throughput perspective, the controller behaves as though it has full information. On the other hand, even at $\rho=1$, the output observation scheme and queue observation scheme still performs at a lower throughput. Finally, for $\rho \in (0,1)$ it is evident that there is a gap in performance for each observation scheme. This gap quantifies the value of information in controlling our model and motivates further analysis.

\vspace{10pt}
\noindent
{\bf Bellman Equations}

It is well-known that the optimal policy that maximizes the throughput follows from the average reward Bellman equations. See \cite{puterman2014markov} for background on Markov Decision Processes (MDP) or \cite{Hernandez2012} for a discussion of average reward optimality with such state spaces. The Bellman equation is then
$$
\mu^*+h(\omega_1,\omega_2)=\max\big\{h^{(1)}(\omega_1,\omega_2),\,h^{(2)}(\omega_1,\omega_2)\big\}\,,
$$
where $h$ is the relative value function and the individual components $h^{(j)}(\cdot,\cdot)$, vary as follows:

\noindent
\textbf{(II) State observation:} 
\begin{align*}
h^{(1)}(\omega_1,\omega_2)&:=r^{(1)}(\omega_1)+\Big[\bar \omega_1\,h\big(p_1, \tau_n^{(2)}(\omega_2)\big)+\omega_1\,h\big(\bar{q}_1, \tau_n^{(2)}(\omega_2)\big)\Big]\,,\\
h^{(2)}(\omega_1,\omega_2)&:=r^{(2)}(\omega_2)+\Big[\bar \omega_2\,h\big(\tau_n^{(1)}(\omega_1), p_2 \big)+\omega_2\,h\big(\tau_n^{(1)}(\omega_1), \bar{q}_2 \big)\Big]\,.
\end{align*}
\noindent
 \textbf{(III) Output  observation:} 
\begin{align*}
h^{(1)}(\omega_1,\omega_2)&:=r^{(1)}(\omega_1)+\Big[\bar r^{(1)}(\omega_1)\,h\big(\tau_f^{(1)}(\omega_1), \tau_n^{(2)}(\omega_2)\big)+r^{(1)}(\omega_1)\,h\big(\tau_s^{(1)}(\omega_1), \tau_n^{(2)}(\omega_2)\big)\Big]\,,\\
h^{(2)}(\omega_1,\omega_2)&:=r^{(2)}(\omega_2)+\Big[\bar r^{(2)}(\omega_2)\,h\big(\tau_n^{(1)}(\omega_1), \tau_f^{(2)}(\omega_2)\big)+r^{(2)}(\omega_2)\,h\big(\tau_n^{(1)}(\omega_1), \tau_s^{(2)}(\omega_2)\big)\Big]\,.
\end{align*}
\noindent
\textbf{(IV) Queue observation:} 
\begin{align*}
h^{(1)}(\omega_1,\omega_2):=r^{(1)}(\omega_1)&+\Big[ \lambda \bar r^{(1)}(\omega_1)\,h\big(\tau_f^{(1)}(\omega_1), \tau_n^{(2)}(\omega_2)\big)+\bar \lambda r^{(1)}(\omega_1)\,h\big(\tau_s^{(1)}(\omega_1), \tau_n^{(2)}(\omega_2)\big)\\
&+\Big(\bar \lambda \bar r^{(1)}(\omega_1)+ \lambda r^{(1)}(\omega_1)\Big)\Big( h(\tau_c^{(1)}(\omega_1), \tau_n^{(2)}(\omega_2))\Big)\Big]\,,
\end{align*}
\begin{align*}
h^{(2)}(\omega_1,\omega_2):=r^{(2)}(\omega_2)&+\Big[\lambda \bar r^{(2)}(\omega_2)\,h\big(\tau_n^{(1)}(\omega_1), \tau_f^{(2)}(\omega_2)\big)+\bar \lambda r^{(2)}(\omega_2)\,h\big(\tau_n^{(1)}(\omega_1), \tau_s^{(2)}(\omega_2)\big)\\
&+\Big(\bar \lambda \bar r^{(1)}(\omega_2)+ \lambda r^{(1)}(\omega_2)\Big)\Big( h(\tau_n^{(1)}(\omega_1), \tau_c^{(2)}(\omega_2))\Big)\Big]\,.
\end{align*}
The optimal decision is then to choose Server~$1$ if and only if $h^{(1)}(\omega_1, \omega_2) \ge h^{(2)}(\omega_1 , \omega_2)$, breaking ties arbitrarily.
Since $\tau_n^{(j)}(0)=\tau_s^{(j)}(0)=\tau_f^{(j)}(0)=p_j $ and $\tau_n^{(j)}(1)=\tau_s^{(j)}(1)=\tau_f^{(j)}(1)=\bar q_j $, for all three aforementioned cases we have:
\begin{equation}\label{eq:points1}
\begin{array}{cc}
\mu^*+h(0,0)=\max\big\{ \mu_0^{(1)}+ h(p_1,p_2), \mu_0^{(2)}+ h(p_1,p_2)\big\},\\
\\
\mu^*+h(1,1)= \max \big\{\mu_1^{(1)}+ h(\bar q_1,\bar q_2),\mu_1^{(2)}+ h(\bar q_1,\bar q_2)\big\}.
\end{array}
\end{equation}
From the ordering in~\eqref{eq:assumption}, the above equations imply that at point $(\omega_1, \omega_2)=(0,0)$, choosing  Server~$1$ is optimal and at point $(\omega_1, \omega_2)=(1,1)$ choosing Server~$ 2$ is optimal. This observation gives some initial insight into the structure of optimal policies. We now purse these further numerically.

\vspace{10pt}
\noindent
{\bf Numerical Investigation of Optimal Policies}

A solution to the above Bellman equations can be obtained numerically using relative value iteration and discretization of the belief state space, $\Omega$. Here we consider that each interval $[0,1]$ for $\omega_1$ and $\omega_2$ is partitioned to 1000 equal sub-intervals. We then run relative value iteration with an accepted error set to $\epsilon=0.0001$. 

Our various numerical experiments indicate the following:
\begin{enumerate}
\item The ordering in \eqref{eq:relation} holds.
\item Increasing (positive) $\rho_j$ always yields an increase in $\mu^*$.
\item Though the myopic policy does not appear to be generally optimal, when both servers are identical, the optimal policy is the symmetric myopic policy. 
\item In all cases, the optimal policy is given by a non-decreasing switching curve within $\Omega$. That is, there exists a function $\omega_2^*(\omega_1)$ where the optimal policy is
\[
\pi(\omega_1, \omega_2) =
\begin{cases}
\mbox{Server 2} &
\mbox{if}\qquad
\omega_2 \ge \omega_2^*\big(\omega_1\big),\\
\mbox{Server 1} & 
\mbox{if}
\qquad
\mbox{otherwise}.
\end{cases}
\]
\item When the ordering in~\eqref{eq:assumption} has strict inequalities, $\omega_2^*(0) > 0$ and $\omega^*_2(1) < 1$.
\item For identical servers, it holds that the switching curve for the output observation case is sandwiched between the switching curve of the state observation case and the myopic switching line \eqref{eq:99}.
\item The switching curve for the queue observation case depends on $\lambda$. Further, when $\lambda$ is at either of the extreme points ($\lambda = 0$ or $\lambda =1$), the queue observation case agrees with the output observation case.
\end{enumerate}

As one illustration of some of the above properties, consider Table~\ref{table1} based on the parameters of \eqref{eq:111} and various values of $\rho_1$ and $\rho_2$. The results in the table further affirm comments $1$ and $3$ above and contains values that agree with Monte Carlo simulation results, similar to those of Figure~\ref{fig:convergence}.

\begin{table}[h!]
 \begin{center}
     \begin{tabular}{ |>{\centering\arraybackslash}  m{1cm} | >{\centering\arraybackslash}  m{1cm} |>{\centering\arraybackslash}  m{2.8cm} |
     >{\centering\arraybackslash}  m{2.8cm} | >{\centering\arraybackslash}  m{2.8cm} | }
    \hline
   \textbf{ $\rho_1$}& \textbf{$\rho_2$}& \textbf{ $\mu^*_\text{state}$} & \textbf{ $\mu^*_\text{output}~$} &  \textbf{$\mu^*_\text{queue}$}\\
    \Xhline{1.5pt}\vspace{1pt}
    
    0.2 & 0.5 &0.5543& 0.5314 &   0.5190\\
     \hline
    0.4& 0.5 &0.5673& 0.5400 &   0.5231 \\
   \hline 
  0.6 & 0.5 & 0.5823& 0.5489 &   0.5289\\
    \hline 
   0.8 & 0.5 & 0.6009  & 0.5647 & 0.5360 \\

              \hline
    \end{tabular}
    \caption{Stability region bounds for observations schemes (II)-(IV) for various $\rho_1 $ and $\rho_2$ values. Note the queue observation case is with $\lambda = 0.5$.}
  \label{table1}  
 \end{center}
\end{table}

As a further illustration, Figure~\ref{fig:Bellman} shows switching curves, $w_2^*(\cdot)$ for the parameters of \eqref{eq:111} with $\rho_1=0.5 $ and $\rho_2=0.7$. In the figure, the red dotted line is the myopic policy line (suboptimal).  The blue solid curve is the switching curve for the output observation case. The green loosely dashed line is related to the queue observation case and the orange densely dashed curve is the switching curve for the state observation case. These curves were obtained by finding the optimal decision for every (discretized) element of $\Omega$ and then observing the switching curve structure.
\begin{figure}[h]
\centering \setlength\figureheight{2.0cm} \setlength\figurewidth{3.0cm}
\includegraphics{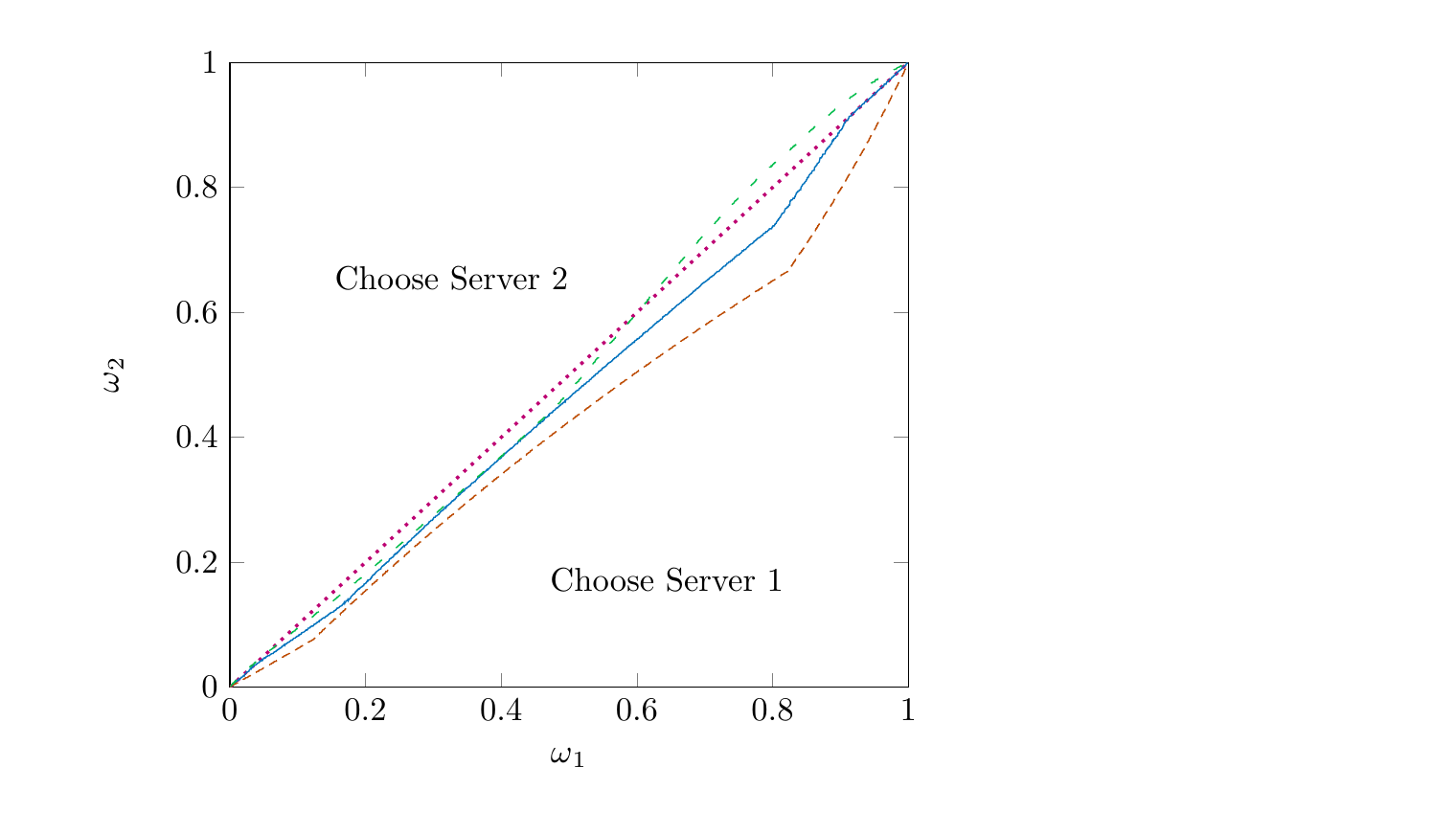}
\begin{center}
\begin{tikzpicture}
\draw[dotted,line width=1.0pt,color=plotcolor2] (0,0)--(0.5,0);
\node at (1.7,0) {\footnotesize Myopic threshold};
\draw[loosely dashed,color=plotcolor3] (3.1,0)--(3.6,0);
\node at (4.8,0) {\footnotesize Opt. threshold   IV};
\draw[solid,color=plotcolor1] (6.3,0)--(6.75,0);
\node at (8,0) {\footnotesize Opt. threshold  III };
\draw[densely dashed,color=plotcolor4] (9.5,0)--(10,0);
\node at (11.1,0) {\footnotesize Opt. threshold  II};
\end{tikzpicture}
\vspace{-2mm}
\caption{Myopic and optimal policies for the state observation Case~II, output observation Case~III and queue observation Case~IV ($\lambda = 0.5$). This is for a system with $\rho_1=0.5$ and $\rho_2=0.7$. 
\label{fig:Bellman}
}\end{center}
\end{figure}
\section{QBD Structured Models for Finite State Controllers}
\label{sec:finiteQBD}
Here we illustrate how Matrix-Analytic Modelling (MAM) can be used to analyse a finite state controller that approximates an optimal controller. Our analysis is for the output observation scheme (Case~IV). Similar analysis can be applied to the other observation schemes as well.

A finite state controller operates by using a finite discrete belief state $\tilde{\Omega}$, representing a discrete grid in $\Omega$. With such a controller, we consider the whole system as a  Quasi-Birth-and-Death (QBD) process (for more details about QBD process see for example \cite{latouche1999introduction}). Using the QBD structure, we find a matrix analytic expression for $\mu^*_{\mbox{output}}$ (denoted by $\mu^*$ in this section).

Take $\tilde{\Omega} = \{1,\ldots,M\}^2$ and define the controller state at time $t$ by $(\psi_1(t),\psi_2(t)) \in \tilde{\Omega}$. In doing so, we treat $\psi_j(t)$ as $\lceil M \omega_j(t) \rceil$. The controller action is (potentially) randomized based on a matrix of probabilities $C$ so that Server~$ 2$ is chosen
with probability $C_{(\psi_1(t),\psi_2(t))}$ and otherwise the choice is Server~$ 1$. That is, the matrix $C$ is a randomized control policy. Such a policy encodes information as in Figure~\ref{fig:Bellman}.

The controller state is updated in a (potentially) randomized manner based on the $M \times M$ stochastic matrices $N^{(j)}, S^{(j)}, F^{(j)}$ for $j=1,2$ as follows:
if Server~$ 1$ was not selected (either because there were no jobs in the queue, or because the other server was selected),
the distribution of the new state is $\big(N^{(1)}_{\psi_1(t),1}, \ldots, N^{(1)}_{\psi_1(t),M} \big)$; that is taken from the row indexed by $\psi_1(t)$.
Similarly, if Server~$1$ was chosen and service was successful ($I=1$), the distribution of the new state is $\big(S^{(1)}_{\psi_1(t),1}, \ldots, S^{(1)}_{\psi_1(t),M} \big)$.
Finally, if Server~$1$  was chosen and the service failed ($I=0$), the distribution of the new state is  $\big(F^{(1)}_{\psi_1(t),1}, \ldots, F^{(1)}_{\psi_1(t),M} \big)$. Similarly, for  Server~$2$, we have $\big(N^{(2)}_{1, \psi_2(t)}, \ldots, N^{(2)}_{M,  \psi_2(t)} \big)$, 
$\big(S^{(2)}_{1, \psi_2(t)}, \ldots, S^{(2)}_{M,  \psi_2(t)} \big)$ and $\big(F^{(2)}_{1, \psi_2(t)}, \ldots, F^{(2)}_{M,  \psi_2(t)} \big)$,  respectively.
Therefore, the rows of matrices $N^{(j)}, S^{(j)}$ and $F^{(j)}$ for $j=1,2$ indicate how to (potentially randomly) choose the next controller state.  Here, S stands for Success, F for Failure and N for No service.

We construct the matrices $N^{(j)}, S^{(j)}, F^{(j)}$ based on a discretization of $\tau_n^{(j)}, \tau_s^{(j)}$ and $\tau_f^{(j)}$, respectively. For example, construction of $S$ from $\tau_s$ is as follows: construct the elements of $S$ by 
\[
\begin{cases}
S_{i,j}=1 & j\, \text{is an integer}   \\
S_{i,\lfloor j \rfloor}=1 &  1 \leq \lfloor j \rfloor \leq M  \\
S_{i,\lceil j \rceil}=1 & otherwise
\end{cases}
\qquad
\mbox{with}
\qquad
j=M \tau_s\big(\frac{i-1}{M}\big),
\]
and $S_{i,k} = 0$ for all other elements with $i, k =1, \ldots, M$. After this, we ensure irreducibility of this matrix  by fixing $\epsilon>0$ (e.g. = $0.001$ as in our numerical examples) and adding $\epsilon/M$ to each elements of the matrix and then renormalizing it. 

The matrices $F$ and $N$ are constructed in a similar way based on $\tau_f$ and $\tau_n$, respectively. 
This is simply a mechanism to encode the transition operators over the finite grid. Hence the matrices $N^{(j)}, S^{(j)}, F^{(j)}$ describe propagation of $\psi_j$ through the belief operators, similarly to the propagation of $\omega$ through their continuous counterparts. 

Now, given such a controller with
\[
\mbox{Controller parameters} = \Big(N^{(1)}, S^{(1)}, F^{(1)}, N^{(2)}, S^{(2)},  F^{(2)}, C\Big),
\]
we construct a Markov chain, $Z(t)$ for the system.
The state of this model at time $t$ is given by the queue length, server environment state, and controller state as follows:
{\small
\[
Z(t) = \Big(\underbrace{Q(t)}_\text{Level}, \underbrace{\big( \overbrace{\big (X_{1}(t),X_{2}(t)\big)}^\text{Servers},\overbrace{(\psi_1(t),\psi_2(t))}^\text{Controller}\!\!\!\big)}_\text{Phase}\Big)
\in
\{0,1,\ldots\} \times \{1,2\}^2\times \{1,\ldots,M\}^2 .
\]
}

\vspace{10pt}
\noindent
{\bf Explicit QBD Construction}

When the states of $Z(t)$ are lexicographically ordered, with first component countably infinite (levels)
and the other components finite (phases), the resulting (infinite) probability transition matrix is of the QBD form:
\begin{equation}
\label{eq:qbd1}
A = \left[
\begin{array}{ccccccc}
\tilde A_0 &  \tilde A_1 &  & \ &  & 0 \\
{A}_{-1} & {A}_0 & {A}_1 &  & &  \\
& {A}_{-1} & {A}_0 & {A}_1 &  &   \\
& & {A}_{-1} & {A}_0 & {A}_1 &    \\
0& &  & \ddots & \ddots & \ddots  &  \\
\end{array}
\right]\,,
\end{equation}
where each of $\tilde A_0$, $\tilde A_1$, $A_{-1}$, $A_0$, and $A_1$ is a block matrix of order $4M^2$ as we construct below. 

The matrix $A_{-1}$ represents the phase transition where there is a one level decrease. Similarly, the matrix $A_{1}$ represents phase transition where there is a one level increase and $A_0$ represents the phase transition where the level remains the same. The blocks are constructed as follows:
\begin{equation}\label{eq:mainmatrices}
\begin{array}{cc}
\tilde A_0 =\overline{\lambda} \tilde{N},\quad 
\tilde A_1=\lambda \tilde{N},\qquad \quad \\
\\
 A_{-1} = \overline{\lambda} \tilde{S},\quad
A_{0} = \overline{\lambda}\tilde{F}+\lambda\tilde{S},\quad
A_1 = \lambda \tilde{F},
\end{array}
\end{equation}
where the matrices, $\tilde{S}, \tilde{F}, \tilde{N}$ (each of order $4 M^2$) denote the change of phase together with a service success, service failure or no service attempt, respectively. For instance, the $(i, j)$-th entry of $\tilde{S}$ is the chance of a service success together with a change of phase from $i$ to $j$ (note that $i$ and $j$ are each 4-tuples). The sum $\tilde{S} + \tilde{F}$ is a stochastic matrix (as is evident from the construction below). Similarly, $\tilde{N}$ is a stochastic matrix. Hence the overall transition probability (infinite) matrix $A$ is stochastic as well.

To construct $\tilde{S}, \tilde{F}$ and $\tilde{N}$, we define $M^2 \times M^2$ matrices $\tilde{S}_{\_k\ell}$, $\tilde{F}_{\_k\ell}$ and $\tilde{N}_{\_k\ell}$ for $k,\ell =0,1$. Taking $\tilde{S}_{\_k\ell}$ as an example, its $(i,j)$-th entry (each represented as a 2-tuple), describes the chance of a success together with a transition of belief state from i to j, when the environment of the first server is in state $k$ and that of the second server is in  state $\ell$. Here $i$ and $j$, each represent the overall system belief state in lexicographic order. That is, we should refer to $i$ as $(i_1,i_2)$ and similarly to $j$. A similar interpretation holds for $\tilde{F}_{\_k\ell}$ and~$\tilde{N}_{\_k\ell}$.
These aforementioned matrices are constructed (for $k, \ell = 0,1$) as follows:
\begin{align*}
\tilde{S}_{\_k\ell}&= \mu_\ell^{(2)}\,\Big(diag\big(vec(C')\big)\Big)\,( N^{(1)} \otimes S^{(2)})+ \mu_k^{(1)}\,\Big(diag\big(vec(\bar C')\big)\Big) \,( S^{(1)} \otimes N^{(2)}),\\
\tilde{F}_{\_k\ell}&=\bar\mu_\ell^{(2)}\,\Big(diag\big(vec(C')\big)\Big)\,( N^{(1)} \otimes F^{(2)})+\bar\mu_k^{(1)}\,\Big(diag\big(vec(\bar C')\big)\Big) \,( F^{(1)} \otimes N^{(2)}), \\
\tilde{N}_{\_k\ell}&=  ( N^{(1)} \otimes N^{(2)}),
\end{align*}
where $diag(\cdot)$ is an operation taking a vector and resulting in a diagonal matrix with the vector in the diagonal, $vec(\cdot)$ is an operation taking a matrix and resulting in a vector with the columns of the matrix stacked up one by one, and $\otimes$ is the standard Kronecker product.
 
To see the above,  let us consider (for e.g.) an element of the matrix $\tilde{S}_{\_k\ell}$ at coordinate $i=(i_1,i_2)$ and $j=(j_1,j_2)$. This describes the probability of the event
\[
W = \{\mbox{Success of service together with a transition to  belief state}~(j_1,j_2)\},
\]
where $X_1 = k$, $X_2=\ell$, $\psi_1 = i_1$, and $\psi_2 = i_2$. 
The event $W$ can be partitioned into $W_1$ (service attempt was on $1$) and $W_2$ (service attempt was on $2$). The chance of $W_2$ is $C_{i_1,i_2}$. With choosing Server~$2$ the success probability is $\mu_\ell^{(2)}$. Then under the event $W_2$, the belief state of Server~$1$ will be updated according to $N^{(1)}$ and the belief state of Server~$2$ with $S^{(2)}$. The $M^2 \times M^2$ matrix $diag(vec(C'))$ is a diagonal matrix where its diagonal elements are the rows of the matrix $C$, each represent the chance of $U=2$.

With the matrices  $\tilde{S}_{\_k\ell}$, $\tilde{F}_{\_k\ell}$ and $\tilde{N}_{\_k\ell}$ (for $k,\ell =0,1$) in hand, we construct the matrices $\tilde{S}$, $\tilde{F }$ and $\tilde{N}$ as:

$$
 \tilde{S}=(P^{(1)}\otimes P^{(2)})\circledast
 \left[ \begin{array}{cccc}
  \tilde{S}_{\_00} & 0 & 0 & 0\\
  0& \tilde{S}_{\_01}& 0 &0\\
 0 & 0& \tilde{S}_{\_10}&0\\
 0&0&0&\tilde{S}_{\_11}
 \end{array} \right],
 $$
 $$
 \tilde{F}=(P^{(1)}\otimes P^{(2)})\circledast
 \left[ \begin{array}{cccc}
 \tilde{F}_{\_00} &0&0&0\\
 0& \tilde{F}_{\_01}&0&0\\
 0&0&\tilde{F}_{\_10}&0\\
 0&0&0& \tilde{F}_{\_11}
  \end{array} \right],
 $$
 $$
  \tilde{N}=(P^{(1)}\otimes P^{(2)})\circledast
  \left[ \begin{array}{cccc}
  \tilde{N}_{\_00} &0&0&0\\
  0& \tilde{N}_{\_01}&0&0\\
  0&0&\tilde{N}_{\_10}&0\\
  0&0&0& \tilde{N}_{\_11}
  \end{array} \right],
$$
where $P^{(j)}$ for $j =1,2$ are the $2\times 2$ probability transition matrices of the servers given by~\eqref{eq:prob.matrix} and operation $\circledast$ is defined as

$$
  \left[ \begin{array}{cc}
  a_{11} & a_{12}\\
 a_{21}&a_{22}
 \end{array}\right]\circledast
 \left[ \begin{array}{cc}
  A & 0\\
 0&B
 \end{array} \right]
 =
 \left[ \begin{array}{cc}
  a_{11}A & a_{12}A\\
 a_{21}B & a_{22}B
 \end{array}\right].
 $$

Putting all of the above components together yields the probability transition matrix of $Z(t)$, $A$.

\vspace{10pt}
\noindent
{\bf Stability Criterion}

A well-known sufficient condition for positive recurrence (stability) of QBDs such as $Z(t)$  is
\[
{\pi}_\infty\big( A_{1} - A_{-1} \big) \mathbf{1} < 0,
\]
where ${\pi}_\infty$ is the stationary distribution of the (finite) stochastic matrix $A_{-1} + A_0 + A_1$ and $\mathbf{1}$ is a column vector of ones. From~\eqref{eq:mainmatrices}, we see that this is also the stationary distribution of $\tilde{S} + \tilde{F}$ which does not depend on $\lambda$.
This property of our QBD allows us to represent the stability criterion as
\begin{equation}
\label{eq:stabQBD}
\lambda < \mu^* =  {\pi}_\infty \tilde{S}{\mathbf 1},
\end{equation}
with $\mu^*$ depending on the controller and system parameters but not depending on $\lambda$. 

In addition to the stability criteria, a further virtue of modelling the system as a QBD is that we can use the vast body of MAM knowledge and algorithms for analysing the system and ultimately optimizing controllers. Nevertheless, our focus in this paper is on stability.

\vspace{10pt}
\noindent
{\bf Numerical Illustration}

We now use our QBD model and the stability criterion \eqref{eq:stabQBD} to explore the performance of finite state controllers.
\begin{figure}[h]
\centering \setlength\figureheight{3.0cm} \setlength\figurewidth{6.0cm}
\input{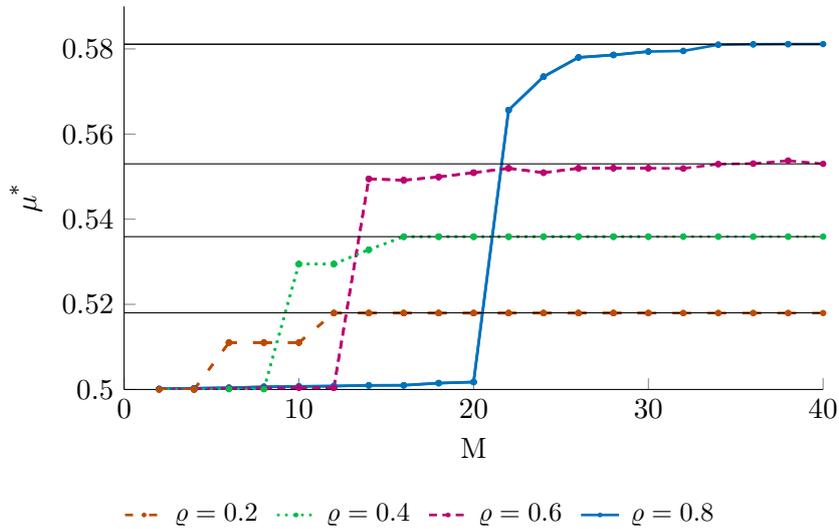}
\begin{center}
\begin{tikzpicture}
\draw[plotstyle4] (0,0)--(0.5,0);
\draw[plotcolor4,fill=plotcolor4,discretemarkers] (0.25,0) circle (\markersize);
\node at (1.2,0) {\small $\rho=0.2$};
\draw[plotstyle3] (2.0,0)--(2.5,0);
\draw[plotcolor3,fill=plotcolor3,discretemarkers] (2.25,0) circle (\markersize);
\node at (3.2,0) {\small $\rho=0.4$};
\draw[plotstyle2] (4.0,0)--(4.5,0);
\draw[plotcolor2,fill=plotcolor2,discretemarkers] (4.25,0) circle (\markersize);
\node at (5.2,0) {\small $\rho=0.6$};
\draw[plotstyle1] (6.0,0)--(6.5,0);
\draw[plotcolor1,fill=plotcolor1,discretemarkers] (6.25,0) circle (\markersize);
\node at (7.2,0) {\small $\rho=0.8$};
\end{tikzpicture}
\end{center}
\vspace{-2mm}
\caption{Stability bound achieved by finite state controllers for observation scheme III with increasing $M$, computed by \eqref{eq:stabQBD}. The limiting horizontal lines are at $\mu^*$ computed by means of relative value iteration of Bellman equations.
\label{fig:MyopicSameBandits}
}
\end{figure}

 In doing so, we consider the parameters as in \eqref{eq:111} with $\rho_1=\rho_2=\rho$. Since in this situation, the servers are identical, the symmetric myopic policy is optimal and we thus restrict attention to a matrix $C$ with
\[
C_{i,j} = 
\begin{cases}
1, & i < j, \\
0.5, & i=j, \\
0, & i > j.
\end{cases}
\]

Using these parameters, we evaluated \eqref{eq:stabQBD} for increasing $M$ and for various values of $\rho$. The results are in Figure~\ref{fig:MyopicSameBandits}. As expected, the performance of the finite state controller converges to that found by numerical solution of the Bellman equations as in the previous section. The sudden increase in performance as $M$ increases (e.g. at $M=20$ for $\rho=0.8$) can be attributed to discretization phenomena. For reference, the values of $\mu^*$ obtained by Bellman equation (as well as the QBD when $M \to \infty$) are   
$0.5179,\, 0.5359,\, 0.5539$ and $0.5815$ for $\rho=0.2,\, 0.4,\, 0.6 $ and $0.8$, respectively. 
\section{Outlook}
\label{sec:outlook}

This paper described some results from a research effort attempting to handle control of stochastic systems with partial observations where the control decision influences the observation made. Explicit analysis of such systems is extremely challenging as is evident by both the complicated Bellman equations and the QBD structure that we put forward (even for a simple system as we consider). Nevertheless, insights obtained on the role of information, e.g. the effect of the observation scheme (I--V) on system stability are of interest.

Our model and numerical results, pave the way for explicit proofs of some of the structural properties outlined above. Moreover, the analysis remains to be extended to more general server environment models, as well as systems with more queues and control decisions. Related work is in \cite{nazarathy2015challenge}, an earlier paper that leads to this work. An aspect in \cite{nazarathy2015challenge} that remains to be further considered is the networked case where the authors investigated (through simulation) cases in which the relationship of stability and throughput is not as immediate as in our current paper. A further related (recent) paper, \cite{meshram2016whittle}, deals with a situation similar to our output observation case (III). In that paper, the authors consider the Whittle index applied to a similar system (without considering a queue and stability). Relating the Whittle index and system stability, is a further avenue that requires investigation.

\section*{Acknowledgement}
AA is supported by Australian Research Council (ARC) grant DE130100291 and would like to acknowledge support from the ARC Centre for Excellence for the Mathematical and Statistical Frontiers (ACEMS).


\end{document}